\documentclass[11pt]{amsart}
\usepackage{amsmath, amssymb, array}
\usepackage[top=30truemm,bottom=30truemm,left=30truemm,right=30truemm]{geometry} 
\usepackage{mathtools}
\usepackage[all]{xy}
\usepackage{ascmac}
\usepackage[dvips]{graphicx}
\usepackage{color}
\usepackage{amscd}
\usepackage{fancyhdr}
\usepackage[bookmarks=false]{hyperref}
\usepackage{amsrefs}
\usepackage{cleveref}
\usepackage{chngcntr}
\usepackage{apptools}
\usepackage{comment}

\usepackage[OT2,OT1]{fontenc} 
\newcommand\cyr 
{
\renewcommand\rmdefault{wncyr} 
\renewcommand\sfdefault{wncyss} 
\renewcommand\encodingdefault{OT2}
 \normalfont 
 \selectfont 
 } 
 \DeclareTextFontCommand{\textcyr}{\cyr}

\newcommand{\disc}{\mathrm{disc}}

\newcommand{\fn}{\footnote}
\newcommand{\fnm}{\footnotemark}

\DeclareMathOperator{\Res}{Res} 

\newmuskip\pFqmuskip

\newcommand*\pFq[6][8]{%
  \begingroup 
  \pFqmuskip=#1mu\relax
  \mathchardef\normalcomma=\mathcode`,
  \mathcode`\,=\string"8000
  \begingroup\lccode`\~=`\,
  \lowercase{\endgroup\let~}\pFqcomma
  {}_{#2}F_{#3}{\left[\genfrac..{0pt}{}{#4}{#5};#6\right]}%
  \endgroup
}
\newcommand{\pFqcomma}{{\normalcomma}\mskip\pFqmuskip}

\title{A generalization of formulas for the discriminants of quasi-orthogonal polynomials with applications to hypergeometric polynomials}

\author{Hideki Matsumura$^{*}$}
\email{hidekimatsumura@keio.jp}
\address{${}^*$Department of Mathematics, Faculty of Science and Technology, Keio University, 3-14-1 Hiyoshi , Kohoku, Yokohama, Kanagawa, Japan}

\thanks{This research is supported by KAKENHI 18H05233.}
 
\date{\today}

\theoremstyle{plain}
 \newtheorem{theorem}{Theorem}[section] 
  \crefname{theorem}{Theorem}{Theorems}
 \newtheorem{proposition}[theorem]{Proposition}
 \crefname{proposition}{Proposition}{Propositions}
 
 \crefname{lemma}{Lemma}{Lemmas}
 \newtheorem{corollary}[theorem]{Corollary}
  \crefname{corollary}{Corollary}{Corollaries}
 
   \crefname{conjecture}{Conjecture}{Conjectures}
 
 \crefname{question}{Question}{Questions}
 
   \crefname{problem}{Problem}{Problems}
 
    \crefname{notation}{Notation}{Notations}
 
\theoremstyle{definition} 
 \newtheorem{definition}[theorem]{Definition}
  \crefname{definition}{Definition}{Definitions}
 \newtheorem{example}[theorem]{Example}
   \crefname{example}{Example}{Examples}
 \newtheorem{remark}[theorem]{Remark}
   \crefname{remark}{Remark}{Remarks}
    
   \crefname{claim}{Claim}{Claims}

\begin{document}


\maketitle

\tableofcontents

\begin{abstract}
Let $K$ be a field.
In this article, we derive a formula for the discriminant of a sequence $\{r_{A,n}+c r_{A,n-1}\}$ of polynomials.
Here, $c \in K$ and $\{r_{A,n} \}$ is a sequence of polynomials satisfying a certain recurrence relation that is considered by Ulas or Turaj.
There are several works calculating the discriminants of given polynomials.
For example, Kaneko--Niiho and Mahlburg--Ono independently proved the formula for the discriminants of certain hypergeometric polynomials that are related to
$j$-invariants of supersingular elliptic curves.
Sawa--Uchida proved the formula for the discriminants of quasi-Jacobi polynomials.
In this article, we present a uniform way to prove a vast generalization of the above formulas.
In the proof, we use the formulas for the resultants $\Res(r_{A,n},r_{A,n-1})$ by Ulas and Turaj that are generalizations of Schur's classical formula for the resultants.
\end{abstract}

\section{Introduction}

Let $K$ be a field and $f(x)$, $g(x) \in K[x]$.
The resultant of $f$ and $g$ is zero if and only if $f$ and $g$ have a common factor of positive degree.
On the other hand, the discriminant of a polynomial $f$ is zero if and only if $f$ has a multiple root.

There are several works calculating resultants and discriminants of given polynomials.
For details, see \cite[\S1]{Turaj}, \cite[\S 1]{Ulas} and the references therein. 
For example, Schur \cite{Schur} proved a formula for the resultant $\Res(r_n,r_{n-1})$ 
of a sequence $\{r_n(x) \}$ of polynomials defined by the following recurrence relation:

$r_0(x)=1$, $r_1(x)=a_1x+b_1$,
\[r_n(x)=(a_nx+b_n)r_{n-1}(x)-c_n r_{n-2}(x) \; (n \geq 2),\]
where, $a_n$, $b_n$, $c_n \in \mathbb{C}$ with $a_n c_n \neq 0$.
Then the resultant of $r_n$ and $r_{n-1}$ is given by
\[\Res(r_n,r_{n-1})=(-1)^{\frac{n(n-1)}{2}} \prod_{i=1}^{n-1} a_i^{2(n-i)}c_{i+1}^i. \]

Schur's result is applied for calculation of the discriminants of orthogonal polynomials and quasi-orthogonal polynomials.
For example, Sawa--Uchida \cite[Theorem 3.1]{SU19} proved the formula for the discriminants of quasi-Jacobi polynomials by using Schur's result and differential equations of Jacobi polynomials.

Ulas \cite[Theorem 3.1]{Ulas} generalized Schur's result to a sequence of polynomials $\{r_{A,n} \}$ defined by a certain recurrence relation, where $A$ is a certain tuple of integers.
He considered the case when $r_{A,n}$ is not necessarily of degree $n$.
His result is recently generalized by Turaj \cite[Theorem 3.1]{Turaj}.

Kaneko--Niiho \cite[Theorem 1]{KN2006} gave the formula for the discriminant of certain hypergeometric polynomials.
The hypergeometric polynomials considered there are rerated to $j$-invariants of supersingular elliptic curves (see \cite{KZ1995}).
A similar formula is independently discovered by Mahlburg--Ono \cite[Theorem 3.1]{MO2004}, where modified polynomials are considered.
They derived the formula for the discriminants of the following hypergeometric polynomials for $r \in \{0,4,6,10 \}$:
\begin{align} \label{MOpoly}
V_r(n;x):=\sum_{m=0}^n c_r(m,n)x^n :=x^n \pFq[4]{2}{1}{-n, n+\beta_r}{\gamma_r}{\frac{2}{x}},
\end{align}
where 
\begin{align*} 
\beta_r:=\frac{r+1}{6}
\end{align*}
and
\begin{align*} 
\gamma_r:=
\begin{cases}
\frac{3}{2} & (r=0,6),\\
\frac{4}{3} & (r=4,10).
\end{cases}
\end{align*}
In \cite{MO2004}, $V_r(n;x)$ is denoted by $B_r(n;x)$.

In the proofs of \cite{KN2006,MO2004}, they used the recurrence formula and contiguous relations of hypergeometric polynomials.

In this article, we present a uniform way to prove a vast generalization of both \cite[Theorem 3.1]{MO2004} and \cite[Theorem 3.1]{SU19}. More precisely,
we prove a formula for the discriminant (\cref{MT}) of a sequence of polynomials of the form $\{r_{A,n}+cr_{A,n-1} \}$, where $c \in K$ and $\{r_{A,n} \}$ is a sequence of polynomials considered by Ulas or Turaj.
This article is organized as follows:

In \S2, we recall the definitions of resultants and discriminants, and
in \S 3, we recall the results of \cite{Ulas,Turaj} (\cref{Ulas,Turaj}).
In \S4, we prove our main theorem (\cref{MT}) by using \cref{Ulas,Turaj}. 
In \S 5, we present some applications to hypergeometric polynomials including \cite[Theorem 3.1]{MO2004}. 

\section{Resultants and discriminants}

In this section, we briefly recall the definitions of resultants and discriminants.

Let $K$ be a field and
\begin{align*}
f(x) &=a_nx^n+ \cdots +a_0 \in K[x],\\
g(x) &= b_mx^m+ \cdots +b_0 \in K[x]
\end{align*}
be polynomials. 

\begin{definition} [{Resultant}]
The resultant of $f$ and $g$ is defined by
\[\Res(f,g):=a_0^mb_0^n \prod_{i=1}^n \prod_{j=1}^m (\alpha_i - \beta_j).\]
Here, $\alpha_1, \ldots, \alpha_n$ are the roots of $f$, and $\beta_1, \ldots, \beta_m$ are the roots of $g$.
\end{definition}

The resultant is the determinant of the $(m+n) \times (m+n)$ matrix
\[\begin{pmatrix}
a_n & a_{n-1} & \hdots & a_0 &   & \\
\text{\huge{0}} & \ddots & \ddots & & \ddots & \text{\huge{0}}\\
& & a_n & a_{n-1} & \hdots & a_0\\
b_m & \hdots &  b_1 & b_0 & & \\
\text{\huge{0}}
& \ddots &  & \ddots & \ddots &  \text{\huge{0}}
\\
 & & b_m & \hdots &  b_1 & b_0
\end{pmatrix}.\]
In the proof of \cref{MT}, we use the following formula:
\begin{proposition} [{cf.\ \cite[eq.\ (5)]{Turaj}}] \label{res}
\[\Res(f,g)=a_n^m \prod_{i=1}^n g(\alpha_i)=(-1)^{nm} b_m^n \prod_{j=1}^m f(\beta_j).\]
\end{proposition}

\begin{definition} [{discriminant}]
The discriminant of a polynomial $f$ is defined by
\[\disc(f):=a_n^{2n-2} \prod_{1 \leq i<j \leq n} (\alpha_i-\alpha_j)^2. \]
Here, $\alpha_1, \ldots, \alpha_n$ are the roots of $f$.
\end{definition}

Discriminants and resultants are related as follows:
\[\disc(f)=\frac{(-1)^{\frac{n(n-1)}{2}}}{a_n} \Res(f,f').\]
Here $f'$ is the derivative of $f$.

\section{Previous works}
In this section, we recall the notations and the formula for the resultants of \cite{Ulas,Turaj}.
Let $K$ be a field.

\subsection{Ulas' formula}

Ulas \cite{Ulas} generalized Schur's formula as follows:

Let $\mathcal{A}:=\{(i,j,k,l) \in \mathbb{Z}_{\geq 0}^4 \mid i \leq j, \; k \geq l \}$.
For $A \in \mathcal{A}$, let $\{r_{A,n}\}$ be the sequence of polynomials on $K$ defined as follows: 
\begin{align*}
r_{A,0}(x) &=\sum_{s=0}^i p_sx^s, \; r_{A,1}(x)=\sum_{s=0}^j q_sx^s, \\
r_{A,n}(x) &= f_n(x) r_{A,n-1}(x)- v_nx^l r_{A,n-2}(x) \; (n \geq 2),
\end{align*}
where
\[f_n(x)=\sum_{s=0}^k a_{n,s}x^s. \]
Here, we assume that $p_s$, $q_s$, $v_n$, $a_{n,s} \in K$.
We also assume that 
$p_iq_ja_{n,k} \neq 0$ for all $n \in \mathbb{Z}_{\geq 2}$ 
and $a_{2,k}q_j - v_2 p_i \neq 0$.
Note that these conditions mean that the leading coefficients of $r_{A,n}(x)$ (resp.\ $f_n(x)$) do not vanish for all $n$ (resp.\ $n \in \mathbb{Z}_{\geq 2}$). 

\cite{Ulas} proved the following formula for the resultant:

\begin{theorem} [{\cite[Theorem 3.1]{Ulas}}] \label{Ulas}
Let $R_n:=\Res(r_{A,n},r_{A,n-1})$, $L_n$ and $C_n$ be the leading coefficient and the constant term of $r_{A,n}$ respectively.
\fn{In what follows, we may drop $A$ from the notation because we fix $A$. For example, we write $R_n$ instead of $R_{A,n}$.}
Then
\begin{enumerate}
\item For all $n \geq 2$, $\deg(r_{A,n})=(n-1)k+j$.
\item \begin{align*}
R_n
=& (-1)^{\sum_{u=2}^n e_A(u)} \left(\prod_{u=2}^n v_u^{(u-2)k+j} L_u^{2k-l} C_u^l \right) R_1\\
=&(-1)^{\sum_{u=2}^n e_A(u)} T_A^{(2k-l)(n-2)}q_0^{l(n-1)}q_j^{k+j-l-i} \left(\prod_{u=0}^{n-2} v_{u+2}^{uk+j} \right) \\
&\cdot \left(\prod_{s=1}^{n-1} a_{s+1,0}^{l(n-s-1)} a_{s+1,k}^{(2k-l)(n-s-1)} \right) R_1, 
\end{align*}
where $e_A(u)=((u-2)k+j)((u-1)k+j+1+l)$ and
\begin{align*}
T_A= \begin{cases}
q_j & (i+l < j+k),\\ 
 \frac{a_{2,k}q_j-v_2p_i}{a_{2,k}} & (i+l = j+k). 
\end{cases}
\end{align*}
\end{enumerate}
\end{theorem}
Note that Schur's formula is the case $A=(0,1,1,0)$.

\begin{remark} \label{ijkl}
If $L_n \neq 0$, then \cref{Ulas} holds under a little more weaker assumption. In fact, it holds for $A=(i,j,k,l) \in \mathbb{Z}_{\geq 0}^{4}$ such that $i \leq j$, $i+l \leq j+k$ and $l \leq 2k$
which we need for application to \cite[Theorem 3.1]{MO2004}. 
Indeed, the condition on $A$ is only used for the calculation of the degree of $r_{A,n}$ ($\deg{r_{A,n}}=(n-1)k+j$ also holds for such $A$ if $L_n \neq 0$) 
and the expression of $L_n$ (see \cite[Proof of Theorem 3.1]{Ulas}).
Note that we do not need an explicit expression of $L_n$ for $l=2k$ since the exponent of $L_n$ in the formula is $0$.
\end{remark}

\begin{remark}
In \cref{Ulas}, we corrected some errors in \cite{Ulas}.

In \cite[Theorem 3.1]{Ulas}, $e(A)$ should be $((u-2)k+j)((u-1)k+j+1\mbox{\boldmath $+l$})$ by the following reason:

By \cref{res}, the right hand side of \cite[eq.\ (2.2)]{Ulas} should be
$(-1)^{nm} b_m^n \prod_{i=1}^m F(\beta_i)$.
Therefore, in the proof of \cite[Theorem 3.1]{Ulas},
\begin{enumerate}
\item p.\ 5, ll. 14--15: 
\[\Res(r_{A,1},x)^l=(-1)^{l \deg(r_{A,1})} r_{A,1}(0)^l=(-1)^{jl}  q_0^l.\]
\item p.\ 5, ll. 22--23: 
\begin{align*}
\Res(r_{A,n-1},-v_n x)^l &=(-1)^{l \deg(r_{A,n-1})} (-v_n)^{\deg(r_{A,n-1})} r_{A,n-1}(0)^l\\
&=(-1)^{((n-2)k+j)l} (-v_n)^{(n-2)k+j} C_{n-1}^l.
\end{align*}
\end{enumerate}
\end{remark}

\subsection{Turaj's formula}
Let $d \in \mathbb{Z}_{\geq 0}$.
Turaj \cite{Turaj} generalized \cref{Ulas} as follows:

Let $\mathcal{A}:=\{(i_0, i_1, \ldots i_d,k,l,m) \in \mathbb{Z}_{\geq 0}^{d+4} \mid i_0 \leq i_1 \leq \ldots \leq i_d , \; k \geq l, \; m \neq 0 \}$.
For $A \in \mathcal{A}$, let $\{r_{A,n}\}$ be the sequence of polynomials on $K$ defined by as follows: 
\begin{align*}
r_{A,0}(x) &=\sum_{s=0}^{i_0} p_{s,0}x^s, \; r_{A,1}(x)=\sum_{s=0}^{i_1} p_{s,1} x^s, \ldots, r_{A,d}(x)=\sum_{s=0}^{i_d} p_{s,d} x^s, \\
r_{A,n}(x) &= g_n(x) r_{A,n-1}^m(x)+\sum_{|\alpha|<m} t_{\alpha,n}(x) {\mathbf r}_{A,n}^{\alpha}(x) r_{A,n-1}(x)+v_nx^l r_{A,n-2}^m(x) \; (n \geq d+1),
\end{align*}
where $\alpha=(\alpha_0, \ldots, \alpha_d) \in \mathbb{Z}_{\geq 0}^{d+1}$, $|\alpha|=\alpha_0+ \cdots +\alpha_d$, $v_n \in K$, 
$g_n(x)$, $t_{\alpha,n}(x)$, ${\mathbf r}_{A,n}^{\alpha}(x) \in K[x]$ and
\begin{align*}
{\mathbf r}_{A,n}^{\alpha}(x) &=r_{A,n-1}^{\alpha_0}(x)r_{A,n-2}^{\alpha_1}(x) \cdots r_{A,n-d-1}^{\alpha_d}(x),\\
g_n(x) &=\sum_{s=0}^k a_{s,n}x^s.
\end{align*}
Assume that
$t_{\alpha,n}(0)=0$,
$\deg(t_{\alpha,n})<\deg(g_n)$ and
$a_{k,n} \prod_{s=0}^{d} p_{i_s,s} \neq 0$ for all $n \in \mathbb{Z}_{\geq 0}$.
Moreover, if $i_d=i_{d-1}$ and $k=l$, then 
$a_{k,d+1}p_{i_d,d}^m + v_{d+1} p_{i_{d-1},d-1}^m \neq 0$.
Note that these conditions mean that the leading coefficients of $g_n(x)$ and $r_{A,n}(x)$ do not vanish for all $n$. 

\begin{theorem} [{\cite[Theorem 3.1]{Turaj}}] \label{Turaj}
\begin{enumerate}
\item For all $n \geq d+1$,
\[\deg(r_{A,n})=k \sum_{s=0}^{n-d-1} m^s+i_d m^{n-d} .\]

\item Let $R_n:=\Res(r_{A,n},r_{A,n-1})$. 
Then
\begin{align*}
R_n=(-1)^{\sum_{s=d+1}^n m^{n-s} e_A(s)}  \left(\prod_{s=d+1}^n L_{s-1}^{\gamma_A(s)} v_s^{\deg(r_{A,s-1})} C_{s-1}^l \right)^{m^{n-s}} R_d^{m^{n-d}},
\end{align*}
where
\begin{align*}
e_A(n) &= (\deg(r_{A,n}) +l) \deg(r_{A,n-1}),\\
\gamma_A(n) &= \deg(r_{A,n})- \deg(v_nx^l r_{A,n-2}) \\
&=\begin{cases}
k-l+m(i_d -i_{d-1}) & (n=d+1),\\
m^{n-d-1}(k+i_d(m-1)) +k-l &(n \geq d+2),
\end{cases}\\
C_n&=\begin{cases}
1 & (l=0),\\
p_{0,d}^{m^{n-d}}  \prod_{s=1}^{n-d} a_{0,d+s}^{m^{n-d-s}}  &(l>0),
\end{cases}\\
L_n&=\begin{cases}
(a_{k,d+1}p_{i_d,d}^m+ v_{d+1}p_{i_{d-1},d-1}^m)^{m^{n-d-1}} \prod_{s=2}^{n-d} a_{k,d+s}^{m^{n-d-s}} & (i_d=i_{d-1} \land k=l),\\
p_{i_d,d}^{m^{n-d}} \prod_{s=1}^{n-d} a_{k,d+s}^{m^{n-d-s}} &(\mbox{otherwise}).
\end{cases}
\end{align*}
\end{enumerate}
\end{theorem}
Note that Schur's formula (resp.\ \cref{Ulas}) is the case $d=1$ and $A=(i,j,k,l,m)=(0,1,1,0,1)$ (resp.\ $d=m=1$ and arbitrary $i,j,k,l$).

\begin{remark}
\begin{enumerate}
\item In the recurrence of Ulas' (resp.\ Turaj's) polynomials, note that the sign before $v_nx^lr_{A,n-2}(x)$ (resp.\ $v_nx^lr_{A,n-2}^m(x)$) is $-1$ (resp.\ $+1$).
When we consider Ulas' (resp.\ Turaj's) polynomials, we follow the 
sign of \cite{Ulas} (resp.\ \cite{Turaj}).

\item In \cref{Turaj}, note that $L_n$ is the leading coefficient of $r_n$.
Moreover, if $l>0$, then $C_n$ is the constant term of $r_n$.
We do not need the precise expression of $C_n$ if $l=0$ since the exponent of $C_n$ in the formula is $0$.

\item
In \cref{Turaj}, we corrected some errors in \cite{Turaj}.
\begin{enumerate}
\item In \cite[Theorem 3.1]{Turaj}, $e(A)$ should be $(\deg(r_{A,n})\mbox{\boldmath $+l$})  \deg(r_{A,n-1})$ by the following reason:

By \cref{res},
in the proof of Theorem 3.1,
\begin{itemize}
\item p. 9 ll.\ 14--15: 
$\Res(r_d,x^l)=(-1)^{l \deg(r_d)} r_d(0)^l=(-1)^{i_dl}  C_d^l$.
\item p. 10 ll.\ 8--9: 
$\Res(r_{n-1},x)^l=(-1)^{l \deg(r_{A,n-1})} C_{n-1}^l$.
\end{itemize}
\item $\prod_{x=1}^{n-d-1}$ in $C_n$ for $l>0$ should be $\prod_{x=1}^{n-d}$.
Indeed, in the last equation in p. 8, $\prod_{s=1}^{n-1}$ should be $\prod_{s=1}^n$.
\end{enumerate}
\end{enumerate}
\end{remark}

\section{Main theorem}

Let $K$ be a field, $c \in K$ and $\{r_{A,n}(x)\}$ be a sequence of polynomials considered in \cite{Ulas,Turaj}.
We call $\{r_{A,n}(x)\}$ in \cite{Ulas} (resp.\ \cite{Turaj}) Ulas' (resp.\ Turaj's) polynomials.
In this section, we compute the discriminant of quasi-Ulas' and quasi-Turaj's polynomials $\{r_{A,n}(x)+cr_{A,n-1}(x)\}$ by using \cref{Ulas,Turaj} under certain differential equations. 
This is a generalization of the formula for the discriminants of quasi-Jacobi polynomials \cite[Theorem 3.1]{SU19}.

Write 
\[r_{A,n;c}(x)=L_n x^{d_n}+ \cdots. \]
and let $y_1,\ldots, y_{d_n}$ be the zeros of $r_{A,n;c}(x)=r_{A,n}(x)+cr_{A,n-1}(x)$.
Here, 
$d_n=(n-1)k+j$ for Ulas' polynomials and 
$d_n=k \sum_{s=0}^{n-d-1} m^s+i_d m^{n-d}$ for Turaj's polynomials.

Suppose that there exist polynomials $F(x)$, $G_{1,n}(x)$, $G_{2,n}(x)$, $H_{1,n}(x)$, $H_{2,n}(x) \in K[x]$ such that
\begin{align*} 
F(x)r_{A,n}'(x) &= G_{1,n}(x)r_{A,n}(x)+G_{2,n}(x)r_{A,n-1}(x)\\
&= H_{1,n}(x)r_{A,n}(x)+H_{2,n}(x)r_{A,n+1}(x).
\end{align*}
We assume that $F(y_t) \neq 0$ for all $1 \leq t \leq d_n$.
For the existence of such $\{r_{A,n}\}$, see \cref{Ulaseg,2F1,MO04}.

Here is our main theorem:

\begin{theorem} \label{MT} 
Let $1 \leq t \leq d_n$ and write
\[-H_{2,n-1}(y_t)c^2+ (H_{1,n-1}(y_t)-G_{1,n}(y_t))c+G_{2,n}(y_t)=\sum_{w=0}^e B_{e-w,n}(c) y_t^w\]
with $B_{1,n}(c), \ldots, B_{e,n}(c) \in K[c]$ for all $t$.
Suppose that
\[\sum_{w=0}^e B_{e-w,n}(c) \xi_{A,n;c,t}^w=0.\]
\begin{enumerate} 
\item 
Let $\{r_{A,n}(x) \}$ be Ulas' polynomials, $L_n$ and $C_n$ be the leading coefficient and the constant term of $r_{A,n}$ respectively. 
Then
\begin{align*}
  \disc(r_{A,n;c}) 
    =&(-1)^{\frac{d_n(d_n+2e-1)}{2}+\sum_{u=2}^n e_A(u)}
     B_{0,n}(c)^{d_n} L_n^{d_n-d_{n-1}-3} \left(\prod_{u=2}^n v_{u+1}^{(u-2)k+j}L_u^{2k-l}C_u^l \right)\\
      &\cdot R_1\prod_{\nu=1}^e r_{A,n;c}(\xi_{A,n;c,\nu}) \prod_{t=1}^{d_n} \frac{1}{F(y_t)}\\
    =&(-1)^{\frac{d_n(d_n+2e-1)}{2}+\sum_{u=2}^n e_A(u)} 
     B_{0,n}(c)^{d_n} L_n^{d_n-d_{n-1}-3} T_A^{(2k-l)(n-2)}q_0^{l(n-1)}q_j^{k+j-l-i} \left(\prod_{u=0}^{n-2} v_{u+2}^{uk+j} \right) \\
&\cdot \left(\prod_{s=1}^{n-1} a_{s+1,0}^{l(n-s-1)} a_{s+1,k}^{(2k-l)(n-s-1)} \right) R_1 
    \prod_{\nu=1}^e r_{A,n;c}(\xi_{A,n;c,\nu}) \prod_{t=1}^{d_n} \frac{1}{F(y_t)}.
    \end{align*}
  Here, 
\begin{align*}
d_n &=(n-1)k+j,\\
e_A(u) &=((u-2)k+j)((u-1)k+j+1+l),\\
T_A&= \begin{cases}
 q_j & (i+l < j+k),\\ 
 \frac{a_{2,k}q_i-v_2p_i}{a_{2,k}} & (i+l = j+k).
\end{cases}
\end{align*}

  \item  
  Let $\{r_{A,n}(x) \}$ be Turaj's polynomials, $L_n$ and $C_n$ be the leading coefficient and the constant term of $r_{A,n}$ respectively. Then
            \begin{align*}
  \disc(r_{A,n;c}) 
    =& (-1)^{\frac{d_n(d_n+2e-1)}{2}+\sum_{s=d+1}^n m^{n-s} e_A(s)}  
      B_{0,n}(c)^{d_n} L_n^{d_n-d_{n-1}-3} \left(\prod_{s=d+1}^n L_{s-1}^{\gamma_A(s)} v_s^{\deg(r_{A,s-1})} C_{s-1}^l \right)^{m^{n-s}} \\
    &\cdot R_d^{m^{n-d}}
   \prod_{\nu=1}^e r_{A,n;c}(\xi_{A,n;c,\nu}) \prod_{t=1}^{d_n} \frac{1}{F(y_t)}.
         \end{align*}
  Here, 
  \begin{align*}
  d_n &=k \sum_{s=0}^{n-d-1} m^s+i_d m^{n-d},\\
e_A(n) &= (d_n +l) d_{n-1},\\
\gamma_A(n) &= d_n-d_{n-2}- \deg(v_nx^l) \\
&=\begin{cases}
k-l+m(i_d -i_{d-1}) & (n=d+1),\\
m^{n-d-1}(k+i_d(m-1)) +k-l &(n \geq d+2).
\end{cases}
\end{align*}
\end{enumerate}
\end{theorem}

\begin{proof}
Since (1) follows from (2), 
\fn{Indeed, by \cite[Corollary 3.2]{Turaj}, \cref{Ulas} follows from \cref{Turaj}.}
we only prove (2). Let $\{r_{A,n}\}$ be Turaj's polynomials.

By definition of discriminant, 
\begin{align*}
 \disc (r_{A,n;c}) &=L_n^{2d_n-2} \prod_{1 \leq i<j \leq d_n} (y_i - y_j)^2  \\
&=(-1)^{\frac{d_n(d_n-1)}{2}} L_n^{d_n-2} \prod_{t=1}^{d_n} r'_{A,n;c} (y_t). 
\end{align*} 

By the differential equation for $r_{A,n}$, 
\begin{align*}
&r'_{A,n;c}(y_t)\\
&=r'_{A,n} (y_t)+cr'_{A,n-1}(y_t)\\
&= F(y_t)^{-1} ((G_{1,n}(y_t)+cH_{2,n-1}(y_t))r_{A,n}(y_t)+ (G_{2,n}(y_t)+cH_{1,n-1}(y_t))r_{A,n-1}(y_t)).
\end{align*}

Since $r_{A,n;c} (y_t)=r_{A,n} (y_t)+cr_{A,n-1} (y_t)=0$,
$r_{A,n} (y_t)=-cr_{A,n-1}(y_t)$. Thus,
\begin{align*}
r'_{A,n;c}(y_t)= F(y_t)^{-1} (-H_{2,n-1}(y_t)c^2+ (H_{1,n-1}(y_t)-G_{1,n}(y_t))c+G_{2,n}(y_t))r_{A,n-1}(y_t).
\end{align*}

Note that by \cref{res},
\[
\tilde{R}_n:=\Res(r_{A,n;c},r_{n-1})=L_n^{d_{n-1}} \prod_{t=1}^{d_n} r_{n-1}(y_t)=R_n
\]
Then we substitute \cref{Turaj} (or \cref{Ulas}).

Suppose that
\[-H_{2,n-1}(y_t)c^2+ (H_{1,n-1}(y_t)-G_{1,n}(y_t))c+G_{2,n}(y_t)=\sum_{w=0}^e B_{e-w,n}(c) y_t^w\]
with $B_{1,n}(c), \ldots, B_{e,n}(c) \in K[c]$. Then 
\[\sum_{w=0}^e B_{e-w,n}(c) y_t^w=(-1)^e B_{0,n}(c) \prod_{\nu=1}^e (\xi_{A,n;c,\nu}-y_t)\]
since 
\[\sum_{w=0}^e B_{e-w,n}(c) \xi_{A,n;c,w}^w=0.\]
Thus, we obtain
 \begin{align*}
  \disc(r_{A,n;c}) 
  &=(-1)^{\frac{d_n(d_n-1)}{2}} L_n^{d_n-2} \prod_{t=1}^{d_n} r'_{A,n;c} (y_t)\\
  &=(-1)^{\frac{d_n(d_n-1)}{2}} L_n^{d_n-2}  \prod_{t=1}^{d_n} \frac{(-1)^e B_{0,n}(c)\prod_{\nu=1}^e (\xi_{A,n;c,\nu}-y_t)}{F(y_t)} r_{A,n-1} (y_t)\\
  &=(-1)^{\frac{d_n(d_n+2e-1)}{2}} 
  B_{0,n}(c)^{d_n} L_n^{d_n-d_{n-1}-3} 
  R_n \prod_{\nu=1}^e r_{A,n;c}(\xi_{A,n;c,\nu}) \prod_{t=1}^{d_n} \frac{1}{F(y_t)}. \fnm
    \end{align*}
     Therefore, the statement follows by \cref{Turaj}.
    \end{proof}
    
    By taking $c \to 0$ in \cref{MT}, we obtain formulas for the discriminants of Ulas' polynomials and Turaj's polynomials.

\section{Examples}
In this section, we present some examples of formula for the discriminants of $\{r_{A,n} \}$ 
involving hypergeometric polynomials.

\subsection{Hypergeometric functions}
In this subsection we recall some properties of hypergeometric functions that we use in the following subsections.

Let $a$, $b$, $c \in \mathbb{R}$ and
\[\pFq[4]{2}{1}{a,b}{c}{x}:=\sum_{k=0}^{\infty} \frac{(a)_k(b)_k}{(c)_k} \frac{x^k}{k!}.\] 
be the hypergeometric function, where, 
\[(\alpha)_k:=\frac{\Gamma(\alpha+k)}{\Gamma(\alpha)}=
\begin{cases}
\alpha (\alpha+1) \cdots (\alpha+k-1) & (k \leq 1),\\
1 & (k=0)
\end{cases} \]
is the Pochhammer symbol.
In the following examples, we make use of the following properties
of hypergeometric functions:

\begin{proposition} 
[Differential equation, {\cite[(9.2.2)]{Lebedev}}] \label{HGde} 
\[\pFq[4]{2}{1}{a,b}{c}{x}'=\frac{ab}{c}\pFq[4]{2}{1}{a+1,b+1}{c+1}{x} \]
\end{proposition}

\begin{proposition} \label{rrde}
\begin{enumerate}
\item \[\pFq[4]{2}{1}{a, b}{c}{x} =\frac{c+(1-a+b)x}{c} \pFq[4]{2}{1}{a, b+1}{c+1}{x} -\frac{(1+b)(1-a+c)x}{(c+1)c} \pFq[4]{2}{1}{a, b+2}{c+2}{x}.\]
\item \[x(1-x)\pFq[4]{2}{1}{a, b}{c}{x}' =\left((c-1) \pFq[4]{2}{1}{a,b-1}{c-1}{x}+(ax+(1-c))\pFq[4]{2}{1}{a,b}{c}{x} \right).\]
\item \[x(1-x)\pFq[4]{2}{1}{a, b}{c}{x}' =bx\pFq[4]{2}{1}{a, b}{c}{x}-\frac{b(c-a)}{c}x \pFq[4]{2}{1}{a, b+1}{c+1}{x}.\]
\item \[x(x-1)\pFq[4]{2}{1}{a, b}{c}{x}' =-\frac{a}{b-1-a}\left((b-c) \pFq[4]{2}{1}{a+1,b-1}{c}{x}+((b-1-a)x-(b-c))\pFq[4]{2}{1}{a,b}{c}{x} \right).\]
\end{enumerate}
\end{proposition}

\begin{proof}
We can prove it by combining \cref{HGde}, contiguous relations (\cite[(9.2.4)--(9.2.6)]{Lebedev}) and an 
symmetry property (\cite[(9.2.1)]{Lebedev}),
or direct calculation.
Here, we prove it by direct calculation.
\begin{enumerate}
\item
 \begin{align*}
& \pFq[4]{2}{1}{a, b}{c}{x} -\frac{c+(1-a+b)x}{c} \pFq[4]{2}{1}{a, b+1}{c+1}{x} +\frac{(1+b)(1-a+c)x}{(c+1)c} \pFq[4]{2}{1}{a, b+2}{c+2}{x}\\
 = & 
 \sum_{k=1}^{\infty} \left(\frac{(a)_k(b)_k}{(c)_k k!}-c\frac{(a)_k(b+1)_k}{(c)_{k+1} k!} -(1-a+b)\frac{(a)_{k-1}(b+1)_{k-1}}{(c)_k (k-1)!}+(1-a+c)\frac{(a)_{k-1}(b+1)_k}{(c)_{k+1} (k-1)!} \right)x^k\\
 = & \sum_{k=1}^{\infty}  \frac{(a)_{k-1}(b+1)_{k-1}}{(c)_{k+1} k!} ((a+k-1)b(c+k)-c(a+k-1)(b+k)-(1-a+b)(c+k)k\\
  &+(1-a+c)(b+k)k)x^k   \\ 
 =& 0. 
 \end{align*}
Note that the constant term is $1-1=0$.
 
\item 
By \cref{HGde},
 \begin{align*}
& x(1-x)\pFq[4]{2}{1}{a, b}{c}{x}' -\left((c-1) \pFq[4]{2}{1}{a,b-1}{c-1}{x}+(ax+(1-c))\pFq[4]{2}{1}{a,b}{c}{x} \right)\\
=& \frac{abx(1-x)}{c}\pFq[4]{2}{1}{a+1, b+1}{c+1}{x} -\left((c-1) \pFq[4]{2}{1}{a,b-1}{c-1}{x}+(ax+(1-c))\pFq[4]{2}{1}{a,b}{c}{x} \right)\\
 = & \sum_{k=2}^{\infty} \left(\frac{(a)_k(b)_k}{(c)_k (k-1)!}-\frac{(a)_{k-1}(b)_{k-1}}{(c)_{k-1} (k-2)!} -\frac{(a)_k(b-1)_k}{(c)_{k-1} k!}-a\frac{(a)_{k-1}(b)_{k-1}}{(c)_{k-1} (k-1)!}
 -(1-c)\frac{(a)_k(b)_k}{(c)_k k!} \right)x^k\\
 = & \sum_{k=2}^{\infty}  \frac{(a)_{k-1}(b)_{k-1}}{(c)_k k!} ((a+k-1)(b+k-1)k-(c+k-1)k(k-1)\\
 &-(a+k-1)(b-1)(c+k-1)-a(c+k-1)k-(1-c)(a+k-1)(b+k-1))x^k   \\ 
 =& 0. 
 \end{align*}
 Note that the constant term is $-((c-1)+(1-c))=0$ and the coefficient of $x$ is
 $ab/c-a(b-1)-a-(1-c)ab/c=0$.
 
\item This follows from (1) and (2).

\item  By \cref{HGde},
\begin{align*}
 & (b-1-a)x(x-1)\pFq[4]{2}{1}{a, b}{c}{x}'+a\left((b-c) \pFq[4]{2}{1}{a+1,b-1}{c}{x}+((b-1-a)x-(b-c))\pFq[4]{2}{1}{a,b}{c}{x} \right)\\
=& \frac{ab(b-1-a)x(x-1)}{c}\pFq[4]{2}{1}{a+1, b+1}{c+1}{x}\\
&+a\left((b-c) \pFq[4]{2}{1}{a+1,b-1}{c}{x}+((b-1-a)x-(b-c))\pFq[4]{2}{1}{a,b}{c}{x} \right)\\
 = & \sum_{k=2}^{\infty} \Large((b-1-a)\frac{(a)_{k-1}(b)_{k-1}}{(c)_{k-1} (k-2)!}-(b-1-a)\frac{(a)_k(b)_k}{(c)_k (k-1)!}+b\frac{(a)_{k+1}(b-1)_k}{(c)_k k!}\\
& -\frac{(a)_{k+1}(b-1)_k}{(c+1)_{k-1} k!}+a(b-1-a)\frac{(a)_{k-1}(b)_{k-1}}{(c)_{k-1} (k-1)!}-a(b-c)\frac{(a)_k(b)_k}{(c)_{k} k!}\Large)x^k\\
 \end{align*}
 \begin{align*}
 = & \sum_{k=2}^{\infty}  \frac{(a)_{k-1}(b)_{k-1}}{(c)_k k!} ((b-1-a)((c+k-1)k(k-1)-(a+k-1)(b+k-1)k+a(c+k-1)k)\\
  &+b(a+k-1)(a+k)(b-1)-(a+k-1)(a+k)(b-1)c-a(b-c)(a+k-1)(b+k-1))x^k   \\ 
 =& 0. 
 \end{align*}
 Note that the constant term is $a((b-c)-(b-c))=0$ and the coefficient of $x$ is
 $-(b-1-a)ab/c+a((b-c)(a+1)(b-1)/c+(b-1-a)-(b-c)ab/c)=0$.
 
\end{enumerate}
\end{proof}

\begin{example} [{Cf.\ \cite[Example 4.5]{Ulas}}] \label{Ulaseg} 
Let
\[V_n(x)=\sum_{i=0}^n \binom{2i}{i} \binom{2(n-i)}{n-i} x^i. \]
Then
\[V_n(x)= \binom{2n}{n}\pFq[4]{2}{1}{1/2, -n}{1/2-n}{x},\]
and by \cref{rrde} (1), 
\begin{align*} 
V_n(x)=\frac{2(2n-1)}{n}(x+1)V_{n-1}(x)-\frac{16(n-1)}{n}xV_{n-2}(x).\\
\end{align*}
Thus, by \cref{Ulas}, 
\fn{In \cite[Example 4.5]{Ulas}, the last exponent should be $2n-s-2$.}
\[R_n:= \Res(V_n,V_{n-1})=2^{3n(n-1)} \prod_{s=1}^{n-1} \left(\frac{s}{2s+1} \right)^s \left(\frac{2s+1}{s+1} \right)^{2n-s-2}. \]
By \cref{rrde} (2) and (3), 
\begin{align*} 
2x(1-x)V'_n(x)=-\frac{n+1}{2} V_{n+1}(x)+ (x+2n+1) V_n(x). 
\\
2x(1-x)V'_n(x)=-2nxV_n(x)+ 8nxV_{n-1}(x).
\end{align*}
Let $F(x)=2x(1-x)$, $G_{1,n}(x)=-2nx$, $G_{2,n}(x)=8nx$, $H_{1,n}(x)=x+2n+1$, $H_{2,n}(x)=-(n+1)/2$. Then
\[-H_{2,n-1}(y_k)c^2+ (H_{1,n-1}(y_k)-G_{1,n}(y_k))c+G_{2,n}(y_k)=((2n+1)c+8n)y_k+\left(\frac{n}{2}c^2+(2n-1)c \right).\]
Let $V_{n;c}(x):=V_n(x)+cV_{n-1}(x)$, $B_{0,n}(c)=(2n+1)c+8n$, 
$B_{1,n}(c)=\frac{n}{2}c^2+(2n-1)c$ and $\xi_{n;c}=-B_{1,n}(c)/B_{0,n}(c)$.
Since $V_{n;c}(0)=V_{n}(0)+cV_{n-1}(0)=\binom{2n}{n}+c\binom{2n-2}{n-1}$, 
$V_{n;c}(1)=V_{n}(1)+cV_{n-1}(1)=4^{n-1}(4+c)$, $d_n=\deg(V_n)=n$ 
and by the proof of \cref{MT}, we obtain
 \begin{align*}
  \disc(V_{n;c}) 
  =&(-1)^{\frac{d_n(d_n+1)}{2}} 
   B_{0,n}(c)^{d_n} L_n^{d_n-d_{n-1}-3} R _n V_{n;c}(\xi_{n;c}) \prod_{k=1}^{d_n} \frac{1}{F(y_k)}\\
   =&(-1)^{\frac{n(n+1)}{2}} 
   ((2n+1)c+8n)^n L_n^{-2}  V_{n;c} \left(-\frac{\frac{n}{2}c^2+(2n-1)c}{(2n+1)c+8n} \right) \\
   & \cdot \left(\prod_{k=1}^{n} \frac{1}{2y_k(1-y_k)}\right)
    \left(2^{3n(n-1)}  \prod_{s=1}^{n-1} \left(\frac{s}{2s+1} \right)^s \left(\frac{2s+1}{s+1} \right)^{2n-s-2}\right)\\
    =&(-1)^{\frac{n(n+1)}{2}} 
   ((2n+1)c+8n)^n L_n^{-2}  V_{n;c} \left(-\frac{\frac{n}{2}c^2+(2n-1)c}{(2n+1)c+8n} \right) \\
   &\cdot \frac{(-1)^n}{2^n\frac{V_{n;c}(0)}{L_n}\frac{V_{n;c}(1)}{L_n}}
    2^{3n(n-1)}  \prod_{s=1}^{n-1} \left(\frac{s}{2s+1} \right)^s \left(\frac{2s+1}{s+1} \right)^{2n-s-2}\\
    =&(-1)^{\frac{n(n-1)}{2}} 2^{3n^2-6n+2}  
    \frac{((2n+1)c+8n)^n}{
    \left(\binom{2n}{n}+c\binom{2n-2}{n-1} \right)(4+c)} V_{n;c} \left(-\frac{\frac{n}{2}c^2+(2n-1)c}{(2n+1)c+8n} \right) \\
   &\cdot \prod_{s=1}^{n-1} \left(\frac{s}{2s+1} \right)^s \left(\frac{2s+1}{s+1} \right)^{2n-s-2} . 
    \end{align*}
\end{example}

\begin{example} \label{2F1}
We can deal with more general hypergeometric polynomials:
Let $a \in K$ and
\[V_n(x):=\pFq[4]{2}{1}{\alpha, \beta-n}{\gamma-n}{x}, \]
with $\alpha$, $\gamma \not\in \mathbb{Z}$ and $\beta \in \mathbb{Z}_{<0}$. 
By \cref{rrde} (1), 
\begin{align*} 
V_n(x)=\frac{\gamma-n+(1-\alpha+\beta-n)x}{\gamma-n} V_{n-1}(x) -\frac{(1+\beta-n)(1-\alpha+\gamma-n)x}{(\gamma+1-n)(\gamma-n)} V_{n-2}(x).\\
\end{align*}
Therefore, by \cref{Ulas},
\begin{align*}
 R_n=& 
(-1)^{\sum_{u=2}^n e_A(u)} \left(\frac{(-1)^{1-\beta}(\alpha)_{1-\beta}}{(\gamma-1)_{1-\beta}} \right)^{n-1} 
\left(\prod_{u=0}^{n-2} \left(\frac{(u-\beta+1)(u+\alpha-\gamma+1)}{(u-\gamma+1)(u-\gamma+2)} \right)^{u+1-\beta} \right)\\
&\cdot \left(\prod_{s=1}^{n-1} \left(\frac{s+\alpha-\beta}{s-\gamma+1} \right)^{n-s-1}\right) R_1\\
=& \left(\frac{(-1)^{1-\beta}(\alpha)_{1-\beta}}{(\gamma-1)_{1-\beta}} \right)^{n-1} 
 \left(\prod_{s=1}^{n-1} \left(\frac{(s-\beta)(s+\alpha-\gamma)}{(s-\gamma)} \right)^{s-\beta}
\frac{(s+\alpha-\beta)^{n-s-1}}{(s-\gamma+1)^{n-1-\beta}}  \right)R_1. \fnm
\end{align*}
Here, we let $s=u+1$ in the second equality.
By \cref{rrde} (2) and (3), 
\begin{align*} 
x(1-x)V_n'(x) =(\gamma-n-1) V_{n+1}(x) +(\alpha x+(1-\gamma+n))V_n(x).\\
x(1-x)V_n'(x) =(\beta-n)xV_n(x)-\frac{(\beta-n)(\gamma-n-\alpha)}{\gamma-n}x V_{n-1}(x).
\end{align*}
Let $F(x)=x(1-x)$, $G_{1,n}(x)=(\beta-n)x$, $G_{2,n}(x)=-(\beta-n)(\gamma-\alpha-n)x/(\gamma-n)$, 
$H_{1,n}(x)=\alpha x+(n-\gamma+1)$, $H_{2,n}(x)=\gamma-n-1$.
Then
\begin{align*}
&-H_{2,n-1}(y_k)c^2+ (H_{1,n-1}(y_k)-G_{1,n}(y_k))c+G_{2,n}(y_k)\\
=&\left((n+\alpha-\beta)c-\frac{(\beta-n)(\gamma-\alpha-n)}{\gamma-n} \right) y_k+((n-\gamma)c^2+(n-\gamma)c).
\end{align*}
Let $V_{n;c}(x):=V_n(x)+cV_{n-1}(x)$, $B_{0,n}(c)=(n+\alpha-\beta)c-\frac{(\beta-n)(\gamma-\alpha-n)}{\gamma-n}$, 
$B_{1,n}(c)=(n-\gamma)c^2+(n-\gamma)c$ and $\xi_{n;c}=-B_{1,n}(c)/B_{0,n}(c)$. Note that $V_{n;c}(0)=V_{n}(0)+cV_{n-1}(0)=1+c$, 
$V_{n;c}(1)=V_{n}(1)+cV_{n-1}(1)$, where $V_n(1)=\sum_{k=0}^{n-\beta} (\alpha)_k (\beta-n)_k/((\gamma-n)_k k!)$,
$d_n:=\deg(V_n)=n-\beta$, $\sum_{u=2}^n e_A(u)=\sum_{u=2}^n (u-1-\beta)(u+2-\beta)= (n - 1) (3 \beta^2 - 3 \beta (n + 3) + n (n + 4))/3 
$ is even.
 \fn{For example, we can prove it by considering $n \mod 6$ and the parity of $\beta$.}
Therefore, by the proof of \cref{MT}, we obtain
 \begin{align*}
  \disc(V_{n;c}) 
  =&(-1)^{\frac{d_n(d_n+1)}{2}} 
   B_{0,n}(c)^{d_n} L_n^{d_n-d_{n-1}-3} R _n V_{n;c}(\xi_{n;c}) \prod_{k=1}^{d_n} \frac{1}{F(y_k)}\\
   =&(-1)^{\frac{(n-\beta)(n-\beta+1)}{2}} 
   \left((n+\alpha-\beta)c-\frac{(\beta-n)(\gamma-\alpha-n)}{\gamma-n} \right)^{n-\beta} L_n^{-2}  \\
   &\cdot V_{n;c} \left(-\frac{(n-\gamma)c^2+(n-\gamma)c}{(n+\alpha-\beta)c-\frac{(\beta-n)(\gamma-\alpha-n)}{\gamma-n}} \right) 
    \prod_{k=1}^{n-\beta} \frac{1}{y_k(1-y_k)}\\
    &\cdot \left(\frac{(-1)^{1-\beta}(\alpha)_{1-\beta}}{(\gamma-1)_{1-\beta}} \right)^{n-1} 
 \left(\prod_{s=1}^{n-1} \left(\frac{(s-\beta)(s+\alpha-\gamma)}{(s-\gamma)} \right)^{s-\beta}
\frac{(s+\alpha-\beta)^{n-s-1}}{(s-\gamma+1)^{n-1-\beta}}  \right)R_1\\ 
    =&(-1)^{\frac{(n-\beta)(n-\beta+1)}{2}} 
   \left((n+\alpha-\beta)c-\frac{(\beta-n)(\gamma-\alpha-n)}{\gamma-n} \right)^{n-\beta} L_n^{-2}  \\
   &\cdot V_{n;c} \left(-\frac{(n-\gamma)c^2+(n-\gamma)c}{(n+\alpha-\beta)c-\frac{(\beta-n)(\gamma-\alpha-n)}{\gamma-n}} \right) 
    \frac{(-1)^{n-\beta}}{\frac{V_{n;c}(0)}{L_n}\frac{V_{n;c}(1)}{L_n}} \left(\frac{(-1)^{1-\beta}(\alpha)_{1-\beta}}{(\gamma-1)_{1-\beta}} \right)^{n-1} \\
    &\cdot \left(\prod_{s=1}^{n-1} \left(\frac{(s-\beta)(s+\alpha-\gamma)}{(s-\gamma)} \right)^{s-\beta}
\frac{(s+\alpha-\beta)^{n-s-1}}{(s-\gamma+1)^{n-1-\beta}}  \right)R_1 \\
    =&(-1)^{\frac{(n-\beta)(n-\beta-1)}{2}}     
    \frac{\left((n+\alpha-\beta)c-\frac{(\beta-n)(\gamma-\alpha-n)}{\gamma-n} \right)^{n-\beta}\left(\frac{(-1)^{1-\beta}(\alpha)_{1-\beta}}{(\gamma-1)_{1-\beta}} \right)^{n-1}}{(1+c)
    \left(\sum_{k=0}^{n-\beta} \frac{(\alpha)_k (\beta-n)_k}{(\gamma-n)_k k!} +c\sum_{k=0}^{n-1-\beta} \frac{(\alpha)_k (\beta-n+1)_k}{(\gamma-n+1)_k k!}  \right)} \\
    &\cdot V_{n;c} \left(-\frac{(n-\gamma)c^2+(n-\gamma)c}{(n+\alpha-\beta)c-\frac{(\beta-n)(\gamma-\alpha-n)}{\gamma-n}} \right) 
    \left(\prod_{s=1}^{n-1} \left(\frac{(s-\beta)(s+\alpha-\gamma)}{(s-\gamma)} \right)^{s-\beta}
\frac{(s+\alpha-\beta)^{n-s-1}}{(s-\gamma+1)^{n-1-\beta}}  \right)R_1. 
    \end{align*}
\end{example}

\subsection{Application to \cite{MO2004}}
In this subsection, we deduce the following formula by Mahlburg and Ono of the discriminants of the hypergeometric polynomials defined by \cref{MOpoly}:

\begin{theorem} [{\cite[Theorem 3.1]{MO2004}}] \label{MO} 
Let $r \in \{0,4,6,10\}$ and $n \geq 1$ be an integer such that $V_r(n;x) \neq 0$. Then
\begin{align*}
\disc(V_r(n;x))=(-1)^{\frac{n(n-1)}{2}}\left(\frac{n(n-\gamma_r+\beta_r)}{2n+\beta_r-1}  \right)^n \frac{c_r(n,0)}{V_r(n;2)} \prod_{j=1}^{n-1} h_r(j)^j c_r(j,0)^2,
\end{align*}
where
\begin{align} \label{hr}
h_r(n) :=- \frac{9n(2n+(-1)^{\frac{r}{2}+1})(12n+r+7)}{(3n+3 \gamma_r)(6n+r+1)(12n+r-5)}.
\end{align}
\end{theorem}

To deduce \cref{MO} from \cref{MT}, we make use of the following properties
of hypergeometric functions:

\begin{proposition}  [{\cite[Proposition 2.2]{MO2004}}] \label{rr} 
For $r \in \{0,4,6,10\}$, we have
\begin{align*}
V_r(n+1;x)=(f_r(n)x+g_r(n))V_r(n;x)+h_r(n) x^2 V_r(n-1;x),
\end{align*}
where
\begin{align*} 
f_r(n) &:= \frac{(12n+r+1)(36n^2+6rn+6n+3 \gamma_r r- 15 \gamma_r)}{(3n+3 \gamma_r)(6n+r+1)(12n+r-5)},\\
g_r(n) &:=-\frac{(12n+r-5)(12n+r+1)(12n+r+7)}{(3n+3 \gamma_r)(6n+r+1)(12n+r-5)}
\end{align*}
and $h_r(n)$ is defined by \cref{hr}.
\end{proposition}

\begin{corollary} \label{de}
\begin{enumerate}
\item 
\begin{align*}
(x-2)x V_r(n;x)'=\frac{n}{2n+\beta_r-1} ((n+\gamma_r-1)xV_r(n;x)+(n+\beta_r-\gamma_r) x^2 V_r(n-1,x)).
\end{align*}
\item 
\begin{align*}
(x-2)x V_r(n;x)'=&\frac{n}{(2n+\beta_r-1)h_r(n)} \\
&\cdot ((((n+\gamma_r-1)h_r(n)-(n+\beta_r-\gamma_r)f_r(n))x-(n+\beta_r-\gamma_r)g_r(n)) V_r(n;x)\\
&+(n+\beta_r-\gamma_r)V_r(n+1;x)).
 \end{align*}
\end{enumerate}
\end{corollary}

\begin{proof}
\begin{enumerate}
\item By \cref{rrde} (4),
\begin{align*}
V_r(n;x)' =&nx^{n-1} \pFq[4]{2}{1}{-n, n+\beta_r}{\gamma_r}{\frac{2}{x}}-2x^{n-2} \pFq[4]{2}{1}{-n, n+\beta_r}{\gamma_r}{\frac{2}{x}} \\
= &\frac{n}{x}V_r(n;x)-\frac{nx^n}{(2n+\beta_r-1)(2-x)} (n+\beta_r-\gamma_r) \pFq[4]{2}{1}{-n+1, n-1+\beta_r}{\gamma_r}{\frac{2}{x}}\\
&-\frac{nx^n}{(2n+\beta_r-1)(2-x)} \left(\frac{2(2n+\beta_r-1)}{x}-(n+\beta_r-\gamma_r)\right)\pFq[4]{2}{1}{-n, n+\beta_r}{\gamma_r}{\frac{2}{x}}\\
= &\frac{n}{x(2-x)}\left((2-x)-\left(2-\frac{n+\beta_r-\gamma_r}{2n+\beta_r-1}x \right)\right)V_r(n;x)
-\frac{n(n+\beta_r-\gamma_r)x}{(2n+\beta_r-1)(2-x)}V_r(n-1;x)\\
=&-\frac{n}{(2-x)} \left(\frac{n+\gamma_r-1}{2n+\beta_r-1}V_r(n;x)+\frac{n+\beta_r-\gamma_r}{2n+\beta_r-1} x V_r(n-1,x) \right).
\end{align*}
Thus, (1) follows.
\item This follows from \cref{rr} and (1).
\end{enumerate}
\end{proof}

\begin{example} \label{MO04}
We deduce \cref{MO} by \cref{MT} (1).
Let
\begin{align*} 
G_{1,n}(x) &=\frac{n(n+\gamma_r-1)}{2n+\beta_r-1}x, \\ 
G_{2,n}(x) &=\frac{n(n+\beta_r-\gamma_r)}{2n+\beta_r-1}x^2, \\
H_{1,n}(x)=&\frac{n}{(2n+\beta_r-1)h_r(n)}  (((n+\gamma_r-1)h_r(n)-(n+\beta_r-\gamma_r)f_r(n))x-(n+\beta_r-\gamma_r)g_r(n)),\\
H_{2,n}(x)=&\frac{n(n+\beta_r-\gamma_r)}{(2n+\beta_r-1)h_r(n)},\\
B_{0,n}(c) =& \frac{n(n+\beta_r-\gamma_r)}{2n+\beta_r-1},\\
B_{1,n}(c) =& \frac{n-1}{(2n+\beta_r-3)h_r(n-1)}  ((n+\gamma_r-2)h_r(n-1)-(n+\beta_r-\gamma_r-1)f_r(n-1))\\
&-\frac{n(n+\gamma_r-1)}{2n+\beta_r-1},\\
B_{2,n}(c) =&-\frac{(n-1)(n+\beta_r-\gamma_r-1)}{(2n+\beta_r-3)h_r(n-1)}(c^2+g_r(n-1)c).
\end{align*}
Then we have
\begin{align*}
-H_{2,n-1}(y_t)c^2+ (H_{1,n-1}(y_t)-G_{1,n}(y_t))c+G_{2,n}(y_t)=\sum_{w=0}^e B_{e-w,n}(c) y_t^w 
\end{align*} 
with $e=2$.
By $B_0(n;x)=1$, $B_1(n;x)=x-2(1+\beta_r)/\gamma_r$,
and by \cref{rr}, 
$A=(i,j,k,l)=(0,1,1,2)$, $p_0=1$, $q_0=-2(1+\beta_r)/\gamma_r=-(r+7)/(3\gamma_r)=g_r(0)$, $q_1=1$,
$a_{n,0}=g_r(n-1)$, $a_{n,1}=f_r(n-1)$, $v_n=-h_r(n-1)$, $d_n=n$, $T_A=q_1=1$, $L_n=1$ and $e_A(u)=(u-1)(u+3)$
in the notation of \cref{MT} (1).
Note that we can apply \cref{MT} (1) to $A=(i,j,k,l)=(0,1,1,2)$ by \cref{ijkl}.
Note also that $R_1=1$ by \cref{res}.
Moreover, by \cref{rr,de}, we take $c \to 0$ in \cref{MT}(1) to obtain
\begin{align*}
  \disc(V_r(n;x)) 
    =&(-1)^{\frac{n(n+3)}{2}+\sum_{u=2}^n (u-1)(u+3)} 
     \left(\frac{n(n+\beta_r-\gamma_r)}{2n+\beta_r-1} \right)^{n}  g_r(0)^{2(n-1)} \prod_{u=0}^{n-2} (-h_r(u+1)^{u+1}) \\
&\cdot \left(\prod_{s=1}^{n-1} g_r(s)^{2(n-s-1)} \right)
     V_r(n;0)^2 \prod_{t=1}^{n} \frac{1}{y_t(y_t-2)}\\
      =&(-1)^{\frac{n(n+3)}{2}+\sum_{u=2}^n (u-1)(u+3)+\sum_{j=1}^{n-1} j} 
     \left(\frac{n(n+\beta_r-\gamma_r)}{2n+\beta_r-1} \right)^{n} \prod_{j=1}^{n-1} h_r(j)^j\\
&\cdot \left(\prod_{s=0}^{n-1} g_r(s)^{2(n-s-1)} \right)
     V_r(n;0)^2 \frac{1}{V_r(n;0)V_r(n;2)}\\
   =&(-1)^{\frac{n(n-1)}{2}}\left(\frac{n(n-\gamma_r+\beta_r)}{2n+\beta_r-1}  \right)^n \frac{c_r(n,0)}{V_r(n;2)} \prod_{j=1}^{n-1} h_r(j)^j c_r(j,0)^2.
    \end{align*}
Here, note that since $c_r(0,0)=1$,
\begin{align*}
V_r(j;0) &= c_r(j,0)\\
&= g_r(j-1)c_r(j-1,0)\\
&= g_r(j-1)g_r(j,2)c_r(j-2,0)\\
&= 
\cdots \\
&= g_r(j-1)\cdots g_r(1)
\end{align*}
 by \cref{rr}. Thus,
 \begin{align*}
\prod_{s=0}^{n-1} g_r(s)^{2(n-s-1)}= 
 &\cdot g_r(0)^2 \cdots g_r(n-3)^2 g_r(n-2)^2\\
 &\cdot g_r(0)^2 \cdots g_r(n-3)^2\\
&\quad  \vdots\\
 &\cdot g_r(0)^2\\ 
 =& c_r(n-1,0)^2 c_r(n-2,0)^2 \cdots c_r(1,0)^2.
 \end{align*}
 Note also that since
  \begin{align*}
\frac{n(n+3)}{2}+\sum_{u=2}^n (u-1)(u+3)=\frac{n(n^2+6n-1)}{3} 
    \end{align*}
 is even, we have
 \begin{align*}
(-1)^{\frac{n(n+3)}{2}+\sum_{u=2}^n (u-1)(u+3)+\sum_{j=1}^{n-1} j}=(-1)^{\frac{n(n-1)}{2}}.
  \end{align*}
  \end{example}
  
\noindent {\bf Acknowledgements.}
The author gratefully thanks Professor Yukihiro Uchida and for helpful comments and discussions.
The author thanks Professor Ken-ichi Bannai for reading the draft and giving helpful comments.
The author also thanks him for warm and constant encouragement.
The author thanks Professor Shuji Yamamoto for helpful comments, discussions and the reference \cite{Lebedev}.
The author thanks Yoshinosuke Hirakawa for the reference \cite{Turaj}.
The author thanks Professor Maciej Ulas for kindly answering the author's questions on \cite{Ulas}.

\end{document}